\documentclass[final,hidelinks,onefignum,onetabnum]{siamart251216}

\usepackage{lipsum}
\usepackage{amsfonts}
\usepackage{graphicx}
\usepackage{epstopdf}
\usepackage{algorithmic}
\usepackage{mathtools}

\usepackage{tcolorbox}
\usepackage{listings}
\usepackage{xcolor}

\lstdefinestyle{matlabcode}{
    language=Matlab,
    basicstyle=\ttfamily\small,
    keywordstyle=\bfseries,
    commentstyle=\itshape,
    showstringspaces=false,
    columns=fullflexible,
    keepspaces=true,
    frame=none,
    numbers=none,
    breaklines=true,
    escapeinside={(*@}{@*)}
}

\usepackage[multiple]{footmisc} 

\hypersetup{
    colorlinks=false,
    pdfborder={0 0 1},
    linkbordercolor=blue,     
    citebordercolor=red,      
    urlbordercolor=magenta,   
    filebordercolor=blue
}

\ifpdf
  \DeclareGraphicsExtensions{.eps,.pdf,.png,.jpg}
\else
  \DeclareGraphicsExtensions{.eps}
\fi

\newsiamremark{remark}{Remark}
\newsiamremark{hypothesis}{Hypothesis}
\crefname{hypothesis}{Hypothesis}{Hypotheses}
\crefname{section}{Section}{Sections}
\newsiamthm{claim}{Claim}
\newsiamremark{fact}{Fact}
\crefname{fact}{Fact}{Facts}

\headers{Converse symplectic interlacing}
{P. Grover, H. K. Mishra, S. Singh, and Temjensangba}

\title{Converse symplectic interlacing via low-rank singular value perturbations
\thanks{\funding{PG acknowledges grant ANRF SERB CRG/2023/000595, and
    HKM acknowledges FRS Project No.~MISC~0147 for financial supports.}}}

\author{Priyanka Grover
\thanks{Department of Mathematics, Shiv Nadar Institution of Eminence, Delhi NCR, India
  (\email{priyanka.grover@snu.edu.in}).}
\and Hemant K. Mishra 
\thanks{Department of Mathematics and Computing, Indian Institute of Technology (ISM) Dhanbad, Jharkhand, India 
  (\email{hemantmishra1124@iitism.ac.in}, \email{temjensangba111@gmail.com}).}
\and Sarishti Singh
\thanks{D\'epartement de math\'ematiques et de statistique, Universit\'e Laval, G1V 0A6, QC, Canada; and D\'epartement \'Education, Universit\'e T\'ELUQ, Qu\'ebec, G1K 9H6, QC, Canada
(\email{sarishtisingh2510@gmail.com}).}
\and Temjensangba \footnotemark[3]}

\usepackage{amsopn}
\DeclareMathOperator{\diag}{diag}

\ifpdf
\hypersetup{
  pdftitle={Converse symplectic interlacing
via low-rank singular value perturbations},
  pdfauthor={P. Grover, H. K. Mishra, S. Singh, and Temjensangba}
}
\fi

\begin{document}

\maketitle

\begin{abstract}
    The symplectic interlacing theorem states that the symplectic spectrum of a $2n \times 2n$ real symmetric positive definite matrix satisfies a $2k$-step interlacing relation with the symplectic spectrum of every $2(n-k) \times 2(n-k)$ $s$-principal submatrix.
    We prove the converse in this paper.
    More precisely, given a global symplectic spectrum of size $n$ and a local symplectic spectrum of size $n-k$ satisfying the $2k$-step interlacing inequalities, we construct a $2n \times 2n$ real symmetric positive definite matrix having the prescribed global symplectic spectrum and a $2(n-k) \times 2(n-k)$ $s$-principal submatrix of it having the prescribed local symplectic spectrum.
    A novelty of our method is building a connection between the $s$-principal submatrix problem and low-rank singular value perturbations.
    We also provide an algorithm to determine a positive definite matrix that realizes the given global and local symplectic spectra.
    In the context of the Gaussian quantum marginal problem, our result implies that the symplectic interlacing inequalities are not merely necessary conditions on Gaussian marginals, but are also the complete spectral compatibility conditions for the global-versus-local-mode marginal problem.
\end{abstract}



\section{Introduction}
 
The symplectic spectral theorem due to Williamson \cite{Williamson1936} states that every $2n \times 2n$ real symmetric positive definite matrix $A$ can be brought to a diagonal form $\Gamma \oplus \Gamma$ via symplectic congruence transformation, where $\Gamma \coloneqq \diag(\mu_1,\ldots,\mu_n)$ and $0<\mu_1 \leq \cdots \leq \mu_n$.
The numbers $\mu_1, \ldots,\mu_n$ are the symplectic eigenvalues of $A$.
Symplectic eigenvalues occur naturally in Hamiltonian dynamics, quantum mechanics, symplectic geometry, and continuous-variable quantum information theory. 
In the last decade, several results for symplectic eigenvalues have been found to be analogous to familiar results for eigenvalues of Hermitian matrices.  
These include variational principles, Weyl and Lidskii type inequalities, Schur--Horn type theorems, perturbation formulas, and Horn type descriptions; see \cite{BhatiaJain2015,BhatiaJainSchurHorn2020,BhatiaJainVariational2021,Jain2021,JainMishra2022,Mishra2020,Paradan2022}.  

Originally published in $1829$, Cauchy~\cite{Cauchy1829} established necessary inequalities satisfied by the eigenvalues of a Hermitian matrix and its principal submatrix, known as Cauchy's interlacing theorem.    
A symplectic analog of Cauchy's interlacing theorem was obtained by Krbek et al.~\cite{KrbekTycVlach2014} in the study of a Gaussian quantum marginal problem, and by Bhatia and Jain~\cite{BhatiaJain2015} in a later work in matrix analysis framework.  Consider an $n$-mode bosonic Gaussian state with covariance matrix $A$ and global symplectic eigenvalues
\[
  \mu_1 \leq \cdots \leq \mu_n.
\]
If $k$ modes are discarded, the covariance matrix of the remaining $(n-k)$ modes is an {\it $s$-principal submatrix} $B$ of $A$, which we will also refer to as a \emph{symplectic compression} of $A$.  
If
\[
  d_1 \leq \cdots \leq d_{n-k}
\]
are the symplectic eigenvalues of $B$, then Krbek et al.~proved the necessary inequalities
\begin{equation} \label{eq:gaussian-k-ineq}
    \begin{aligned}
      \mu_j&\leq d_j\leq\mu_{j+2k}, &&1\leq j\leq n-2k,\\
      \mu_j&\leq d_j, &&n-2k<j\leq n-k.
    \end{aligned}
    \end{equation}
When $n-2k \leq 0$, only the second family is present.
They also proved a stronger interlacing relation for passive Gaussian transformations and established the converse in that restricted setting.  
They provided some numerical evidence in support of the converse for the special case $k=1$ and stated it as an open problem.   
    A related problem to marginals for multipartite Gaussian quantum systems was studied by Eisert et al.~\cite{EisertTycRudolphSanders2008}.
    The authors established necessary and sufficient conditions for two sets of positive numbers $\alpha_1, \ldots, \alpha_n$ and $\beta_1, \ldots, \beta_n$ for which there exists a covariance matrix with global symplectic spectrum $\alpha_1,\ldots, \alpha_n$ and each of its $k$th one-mode marginal covariance matrix has symplectic spectrum $\beta_k$ for $1 \leq k \leq n$.
    This is a symplectic analog of the classical Sing--Thompson theorem \cite{thompson1977singular}.
    The special case $n=2$ of the aforementioned result affirms the converse of the symplectic interlacing theorem for two-mode systems.

    In $1957$, Fan and Pall~\cite{FanPall1957} proved the sufficiency of the inequalities in the Cauchy's interlacing theorem.
    In this paper, we establish a symplectic analog of the result of Fan and Pall and hence
    settle the converse problem of the symplectic interlacing theorem in full generality.
More concretely, given two spectra
\[
  0<\mu_1\leq\cdots\leq\mu_n,
  \qquad
  0<d_1\leq\cdots\leq d_{n-k}
\]
satisfying \eqref{eq:gaussian-k-ineq}, we provide an explicit construction of a $2n \times 2n$ real symmetric positive definite matrix $A$ whose symplectic eigenvalues are $\mu_1, \ldots, \mu_n$ and which has a $2(n-k) \times 2(n-k)$ $s$-principal submatrix with symplectic eigenvalues $d_1, \ldots, d_{n-k}$.
We also emphasize that the matrix realizing the prescribed global and local symplectic spectra can be chosen in a particularly simple form.  If
\[
  \Lambda \coloneqq \diag(\mu_n,\ldots,\mu_1),
\]
then there exists $X \in \mathrm{GL}_n(\mathbb{R})$, where $\mathrm{GL}_n(\mathbb{R})$ denotes the general linear group of invertible $n \times n$ real matrices, such that
\begin{equation} \label{eq:intro-point-orbit}
  A_X \coloneqq (X \Lambda X^{\top}) \oplus(X^{-\top} \Lambda X^{-1})
\end{equation}
has the prescribed global and local symplectic spectra.
 Following the terminology used in \cite{KrbekTycVlach2014}, such symplectic transformations are point transformations.
\medskip

One of the key ideas behind proving the main result is building a connection between the symplectic compression problem and low-rank singular value perturbations, described as follows

\begin{description}

\item[(a) Bridge between symplectic eigenvalues and singular values:]
\sloppy The symplectic spectrum of an $s$-principal submatrix obtained by deleting $k$ canonical pairs from a matrix obtained by point transformation is encoded by the singular values of an $n \times n$ defect matrix
\[
M\coloneqq \Lambda-UV^{\top},
\qquad
U,V\in M_{n\times k}(\mathbb{R}),
\qquad
V^{\top}\Lambda^{-1}U=I_k.
\]
More precisely, the nonzero singular values of $M$ are exactly the symplectic eigenvalues of the retained $s$-principal submatrix; see \cref{prop:the connection}. This reduces the inverse symplectic interlacing problem to a rank-$k$ singular value reassignment problem subject to the normalization
\[
V^{\top}\Lambda^{-1}U=I_k.
\]

\item[(b) Reconstruction of a point transformation from the normalization:]
Conversely, the normalization above is precisely what is needed to recover a point transformation. Let $F \in \mathbb{R}^{n \times k}$ have as its columns any $k$ distinct standard basis vectors of $\mathbb{R}^n$, and let $\Lambda>0$ be diagonal. For every $U,V \in \mathbb{R}^{n \times k}$ satisfying
\[
V^{\top}\Lambda^{-1}U=I_k,
\]
there exists $X\in\operatorname{GL}_n(\mathbb{R})$ such that
\[
U=\Lambda^{1/2}X^{\top}F,
\qquad
V=\Lambda^{1/2}X^{-1}F;
\]
see \cref{lem:reconstruct-X}. Thus, every normalized low-rank defect arising in the singular value reassignment problem can be lifted back to an $s$-principal compression on a point-symplectic orbit.

\end{description}

In the context of Gaussian quantum information theory, the converse has a direct physical interpretation for multimode electromagnetic fields. The symplectic eigenvalues of a Gaussian covariance matrix can be identified, up to the usual normalization, with the mean photon numbers of the thermal modes appearing in its Williamson normal form. Consequently, our result provides a complete characterization of the possible relations between the average photon numbers of the thermal modes of the global electromagnetic-field system and those of a subsystem obtained by discarding one or more modes. In particular, every collection of subsystem mean photon numbers satisfying the symplectic interlacing inequalities can be realized by a Gaussian state with the prescribed global mean photon numbers.

The organization of the paper is as follows.
In \cref{sec:prelim}, we review a basic symplectic spectral framework, define $s$-principal submatrices, and describe point-symplectic orbits.
We establish some results related to \emph{s}-principal compression and low-rank perturbations of matrices in \cref{sec:s-compression}.
The main result, the converse of symplectic interlacing, is presented in \cref{sec:main_result}.
In \cref{sec:construction}, we provide an algorithm to explicitly construct the positive definite matrix that realizes the given global and local symplectic spectra.
Finally, an application of our main result to Gaussian quantum information theory is discussed in \cref{sec:gaussian-application}.

\section{Symplectic eigenvalues and $s$-principal submatrices} \label{sec:prelim}

We denote by $M_{m \times n}(\mathbb{R})$ the set of $m \times n$ real matrices, and we use the notation $M_n(\mathbb{R}) \coloneqq M_{n \times n}(\mathbb{R})$.
Let $J\coloneqq \begin{bsmallmatrix}
    0&I_n\\-I_n&0
\end{bsmallmatrix}$ be the standard symplectic matrix, where $I_n$ is the identity matrix of size $n \times n$.
The real symplectic group is
\begin{equation*}
    \mathrm{SP}_{2n}(\mathbb{R}) \coloneqq \{S\in M_{2n}(\mathbb{R}):S^{\top}JS=J\}.
\end{equation*}

For a symmetric positive definite matrix $A \in M_{2n}(\mathbb{R})$, Williamson's theorem \cite{Williamson1936,SimonChaturvediSrinivasan1999} gives $S \in \mathrm{SP}_{2n}(\mathbb{R})$ such that
\begin{equation} \label{eq:Williamson}
  S^{\top}AS=\Gamma \oplus \Gamma,
  \qquad \Gamma \coloneqq \diag(\mu_1,\ldots,\mu_n).
\end{equation}
We write
\[
  \operatorname{sspec}(A)\coloneqq(\mu_1,\ldots,\mu_n),
  \qquad 0<\mu_1\leq\cdots\leq\mu_n,
\]
for the ordered symplectic spectrum of $A$.  Equivalently, the eigenvalues of $JA$ are $\pm i\mu_j$ \cite{JainMishra2022}.

We next describe the principal submatrices relevant to symplectic interlacing.  Let
\[
  \mathcal{I} \coloneqq\{i_1<\cdots<i_r\}\subseteq\{1,\ldots,n\}
\]
and let
\begin{equation} \label{eq:matrix-retaining-indices}
  E_\mathcal{I} \coloneqq[e_{i_1}\ \cdots\ e_{i_r}]\in M_{n \times r}(\mathbb{R}),
\end{equation}
where $e_j$ is the $j$th standard basis vector of $\mathbb{R}^n$.
The matrix
\begin{equation} \label{eq:sprincipal}
  A_\mathcal{I} \coloneqq \left(E_\mathcal{I}^\top \oplus E_\mathcal{I}^\top \right) A \left(E_\mathcal{I} \oplus E_\mathcal{I} \right)
\end{equation}
is called the $s$-principal submatrix of $A$ associated with $\mathcal{I}$.
Thus an $s$-principal submatrix is obtained by retaining the same set of indices in the two $n \times n$ coordinate blocks determined by $\mathcal{I}$.
In particular, deleting one mode means deleting the same one index from each of these two blocks.

A special class of symplectic matrices will play an important role.  Given $X \in \mathrm{GL}_n(\mathbb{R})$, set
\begin{equation} \label{eq:point-symp}
  \mathcal{S}_X \coloneqq X\oplus X^{-\top}.
\end{equation}
A direct calculation gives $\mathcal{S}_X^{\top} J \mathcal{S}_X=J$.
For an $n \times n$ positive diagonal matrix $\Lambda$, we define $A_X$ as
\begin{equation} \label{eq:GammaX}
  A_X \coloneqq \mathcal{S}_X (\Lambda \oplus \Lambda) \mathcal{S}_X^{\top}
  = (X \Lambda X^{\top}) \oplus (X^{-\top} \Lambda X^{-1}).
\end{equation}
We call the family of matrices $A_X$, as $X$ ranges over $\mathrm{GL}_n(\mathbb{R})$, the \emph{point-symplectic orbit} of $\Lambda \oplus \Lambda$.
It is easy to check that the point-symplectic orbit of $\Lambda \oplus \Lambda$ is independent of the ordering of the diagonal entries of $\Lambda$.
Every matrix on this orbit has symplectic spectrum given by the diagonal entries of $\Lambda$.
Our main result shows that this restricted orbit is already sufficient to realize every $s$-principal spectrum allowed by the symplectic interlacing inequalities \eqref{eq:gaussian-k-ineq}.

\section{Symplectic spectrum of $s$-principal compression and singular values of low-rank perturbation} \label{sec:s-compression}

    The main aim of this section is to establish correspondence between the symplectic spectrum of an $s$-principal submatrix and the singular values of a low-rank perturbation of a matrix corresponding to the given matrix.
    This allows us to leverage the classical results of singular values interlacing inequalities given in \cite{ChuChu2006}.

    We begin with a simple observation on block-diagonal positive definite matrices.  
    If $X$ is a matrix with real eigenvalues, we shall denote by $\lambda(X)$ the vector consisting of the eigenvalues of the matrix $X$ with the entries arranged in the non-decreasing order. 
    Let $\sigma(X)$ denote the vector of singular values of the matrix $X$ with its entries arranged in the non-decreasing order.

\begin{lemma} \label{lem:blockdiag}
    Let $P, Q \in M_n(\mathbb{R})$ be symmetric positive definite matrices. 
    The symplectic eigenvalues of $P \oplus Q$ are the positive square roots of the eigenvalues of $PQ$.
    Thus, the symplectic eigenvalues of $P \oplus Q$ are the singular values of $P^{1/2}Q^{1/2}$.
\end{lemma}

\begin{proof}
First, we observe that $PQ$ is similar to the symmetric positive definite matrix $P^{1/2} Q P^{1/2}$.
Thus, all the eigenvalues of $PQ$ are positive.
We have
\begin{equation*}
\left(i J (P \oplus Q) \right)^2 = QP \oplus PQ .
\end{equation*}
This means that the eigenvalues of $(i J (P \oplus Q))^2$ are the eigenvalues of $PQ$, each counted twice.
It immediately follows that the squares of the symplectic eigenvalues of $P \oplus Q$ are the eigenvalues of $PQ$.
This proves the first part.
The second part follows from the first and by observing that
\begin{equation*}
\sqrt{\lambda(PQ)}
              = \sqrt{\lambda \left( P^{1/2} Q P^{1/2} \right)}
              = \sqrt{\lambda \left( (P^{1/2} Q^{1/2}) (P^{1/2} Q^{1/2})^\top \right)}
              = \sigma \left( P^{1/2} Q^{1/2} \right). 
\end{equation*}
\end{proof}

Let $\Lambda \in M_n(\mathbb{R})$ be a positive diagonal matrix and let $X \in \mathrm{GL}_n(\mathbb{R})$.
Let $A_X$ be the matrix as defined in \eqref{eq:GammaX}.
For any $k$ indices $1 \leq i_1 < \cdots < i_k \leq n$, delete the rows and columns of $X \Lambda X^\top$ and $X^{-\top} \Lambda X^{-1}$ corresponding to those $k$ indices.
If $m = n-k$, let the $2m \times 2m$ retained \emph{s}-principal submatrix be denoted by $B_X$.
Let $E \in M_{n \times m}(\mathbb{R})$ be the matrix whose columns are the standard basis vectors corresponding to the retained indices, and let $F \in M_{n \times k}(\mathbb{R})$ be the matrix whose columns are the standard basis vectors corresponding to the discarded indices.
Notice that
\begin{equation} \label{eq: relationship between E and F}
    EE^\top + FF^\top = I_n \text{ and } E^\top F = 0.
\end{equation}
Let 
\begin{equation}
   P_E\coloneqq E^{\top}X\Lambda X^{\top}E
  \text{ and }
  Q_E\coloneqq E^{\top}X^{-\top}\Lambda X^{-1}E
\end{equation}
so that
\begin{equation}\label{eq:marginal-PQ}
  B_X \coloneqq P_E \oplus Q_E.
\end{equation}
Define
\begin{equation}\label{eq:CD}
  C \coloneqq \Lambda^{1/2} X^{\top}E,
  \text{ and }
  D \coloneqq \Lambda^{1/2}X^{-1}E.
\end{equation}
We observe that $P_E = C^{\top}C$ and $Q_E=D^{\top}D$.

\begin{proposition} \label{prop:the connection}
With the notation above, the non-zero singular values of the $n\times n$ matrix
\begin{equation}\label{eq:M-defect-k}
  M_X\coloneqq CD^{\top}
\end{equation}
are exactly the symplectic eigenvalues of $B_X$.  Moreover,
\begin{equation}\label{eq:M-rankk}
  M_X=\Lambda-UV^{\top},
\end{equation}
where $U,V\in M_{n \times k}(\mathbb{R})$ are given by
\begin{equation}\label{eq:UV}
  U\coloneqq \Lambda^{1/2}X^{\top}F,
  \qquad
  V\coloneqq \Lambda^{1/2}X^{-1}F,
\end{equation}
and
\begin{equation}\label{eq:normalization-k}
  V^{\top}\Lambda^{-1}U=I_k.
\end{equation}
\end{proposition}
Moreover, $\operatorname{nullity} M_X = k$.

\begin{proof}
From \cref{lem:blockdiag}, we have
\begin{equation} \label{eq:sspec-sigma-relation-1}
( \underbrace{0,\ldots, 0}_{k \text{ entries}}, \underbrace{\operatorname{sspec} (B_X)}_{m \text{ entries}})
= \left( 0, \sqrt{\lambda(P_E Q_E)} \right)
= \left( 0, \sqrt{\lambda(C^\top C D^\top D)} \right).
\end{equation}
We recall that if $M$ and $N$ are matrices of sizes $n_1 \times n_2$ and $n_2 \times n_1$ with $n_1 < n_2$, respectively, then all the $n_1$ eigenvalues of $MN$ are also eigenvalues of $NM$, and the remaining $n_2-n_1$ eigenvalues of $NM$ are all zeros.
Therefore, 
\begin{equation} \label{eq:sspec-sigma-relation-2}
\left( 0, \sqrt{\lambda(C^\top C D^\top D)}\right)
=  \sqrt{ \lambda \left( C D^\top D C^\top \right) }
= \sqrt{ \lambda \left( (CD^\top)(CD^\top)^\top \right) }
= \sigma \left( CD^\top \right).
\end{equation}
The relations \eqref{eq:sspec-sigma-relation-1} and \eqref{eq:sspec-sigma-relation-2} imply that the non-zero singular values of $M_X$ are the symplectic eigenvalues of $B_X$.

We have 
\begin{align*}
  M_X
  &=\Lambda^{1/2}X^{\top}EE^{\top}X^{-\top}\Lambda^{1/2}\\
  &=\Lambda^{1/2}X^{\top}(I-FF^{\top})X^{-\top}\Lambda^{1/2}\\
  &=\Lambda
    -\Lambda^{1/2}X^{\top}F F^{\top}X^{-\top}\Lambda^{1/2}\\
  &=\Lambda-UV^{\top},
\end{align*}
where $ U =  \Lambda^{1/2} X^{\top} F$ and $V = \Lambda^{1/2} X^{-1} F$ as desired.
It is easy to verify that
\[
 V^{\top}\Lambda^{-1}U
 =F^{\top}X^{-\top}\Lambda^{1/2}\Lambda^{-1}\Lambda^{1/2}X^{\top}F
 =F^{\top}F=I_k.
\]
Finally, $\operatorname{nullity} M_X = k$ follows from the fact that the number of non-zero singular values of a matrix equals its rank.
\end{proof}

The previous result reveals that the symplectic spectrum of an $s$-principal submatrix obtained by deleting $k$ canonical pairs on a point-symplectic orbit is encoded in the singular values of a rank-$k$ perturbation of the Williamson diagonal whose nullity is exactly $k$.
 The next lemma shows that the normalization \eqref{eq:normalization-k} appearing in the previous result is exactly what is needed to reconstruct a point transformation.

\begin{lemma}\label{lem:reconstruct-X}
Let $\Lambda>0$ be diagonal and let $U,V \in M_{n \times k}(\mathbb{R})$ have full column rank and satisfy
$V^{\top} \Lambda^{-1} U = I_k.$
Let $F \in M_{n \times k}(\mathbb{R})$ be any matrix whose columns are $k$ distinct standard basis vectors of $\mathbb{R}^n$.
Then there exists $X\in\mathrm{GL}_n(\mathbb{R})$ such that
\begin{equation}
  U = \Lambda^{1/2} X^{\top} F,
  \qquad
  V = \Lambda^{1/2} X^{-1} F.
\end{equation}
\end{lemma}

\begin{proof}
Let $P \coloneqq \Lambda^{-1/2}U$ and $Q \coloneqq \Lambda^{-1/2}V$.
Then, $\operatorname{rank} Q^\top = \operatorname{rank} Q = k$ implying that $\operatorname{nullity} Q^\top = n- k$. 
Choose an $R \in M_{(n-k) \times n}(\mathbb{R})$ whose rows form a basis of $\ker(Q^{\top})$, and define
\begin{equation}\label{eq:X-reconstruction}
  Y \coloneqq \begin{bmatrix} R \\ P^{\top} \end{bmatrix}.
\end{equation}
Since $RQ=0$ and $P^{\top}Q=I_k$,
\begin{equation} \label{eq:definition of H}
  Y Q = \begin{bmatrix}
      0 \\ I_k
  \end{bmatrix} \eqqcolon H, \text{ say }.  
\end{equation}
There exists some $n \times n$ permutation matrix $\Pi$ such that $\Pi Y Q = \Pi H = F$.
Putting $X \coloneqq \Pi Y$, we obtain $XQ = F$.
Assuming the invertibility of $X$, we get $X^{-1}F = Q$ giving us $V = \Lambda^{1/2} X^{-1} F$.
It again follows from \eqref{eq:X-reconstruction} and \eqref{eq:definition of H} that $Y^\top H = P$.
Then $P = Y^\top H = Y^\top \Pi^\top \Pi H = X^\top F$.
This immediately gives $U = \Lambda^{1/2} X^\top F$.

It remains to show that $X$ is invertible for which it suffices to show that $Y$ is invertible.
Let $x \in \mathbb{R}^n$ such that $Yx = 0$.
Then $Rx=0$ implying that $x$ lies in the range of $Q$.
Say, $x = Qy$ for some $y \in \mathbb{R}^k$.
Also, $Yx = 0$ implies $P^\top x = 0$.
Since $P^\top Q = I_k$, it then follows that $y = 0$, which in turn gives $x = 0$.
This proves the invertibility of $Y$.
\end{proof}

The normalization in \cref{lem:reconstruct-X} is automatic for a low-rank update that creates exactly $k$ zero singular values.

\begin{lemma} \label{lem:nullity-normalization}
Let $\Lambda>0$ and suppose
\begin{equation*}
    M=\Lambda-UV^{\top}
\end{equation*}
with $U, V \in M_{n \times k}(\mathbb{R})$ of full column rank.
If $\operatorname{nullity} M=k$, then
\begin{equation} \label{eq:nullity implies equality with I_k}
   V^{\top} \Lambda^{-1} U = I_k. 
\end{equation}
\end{lemma}

\begin{proof}
 Since $\operatorname{nullity} M = k$, choose $Z \in M_{n \times k}(\mathbb{R})$ whose columns form a basis of $\ker M$.
 Note that $\operatorname{rank} Z = k$.
 Now, $MZ = 0$ gives 
 \begin{equation} \label{eq:LX}
    Z = \Lambda^{-1}UV^\top Z.  
 \end{equation}
 This implies
 \begin{align}
     k = \operatorname{rank}(Z) = \operatorname{rank}(\Lambda^{-1}UV^\top Z) = \operatorname{rank}(UV^\top Z) \leq \operatorname{rank}(V^\top Z) \leq k.
 \end{align}
 This means $\operatorname{rank}(V^\top Z)=k$ and hence it is invertible.
 Multiplying both sides of \eqref{eq:LX} on the left by $V^\top$, we obtain $V^\top Z = V^\top \Lambda^{-1} U V^\top Z$.
 Since $V^\top Z$ is invertible, we get that $V^\top \Lambda^{-1} U=I_k$.
\end{proof}

\section{The converse of the symplectic interlacing theorem} \label{sec:main_result}
    Krbek et al.~\cite{KrbekTycVlach2014} presented numerical evidence in support of the converse of the symplectic interlacing theorem in the one-mode case, i.e., $k=1$.  
    In particular, their experiments with point transformations suggest that this restricted family of symplectic transformations may already be sufficient to realize the global and local symplectic spectra satisfying the interlacing inequalities \eqref{eq:gaussian-k-ineq}. 
    The aim of this section is to establish the converse of the symplectic interlacing theorem in full generality.
    Moreover, we also show that the realizing matrix may always be chosen on the point-symplectic orbit \eqref{eq:GammaX}.

     We first recall a result on the interlacing inequalities for singular values given in \cite[Theorem~3.1]{ChuChu2006} that will be important in proving the main result.
     Although their result is stated for rectangular real matrices, we require only the square case of their result. 

    \begin{theorem} [\cite{ChuChu2006}] \label{thm:paramount utility}
    Consider $\alpha_1 \geq \cdots \geq \alpha_n \geq 0$ and $\beta_1 \geq \cdots \geq \beta_n \geq 0$.
    Let $A \in M_n(\mathbb{R})$ having $\sigma(A) = ( \alpha_n, \ldots, \alpha_1 )$.
    For an integer $0 < k \leq n$, there exists $G \in M_n(\mathbb{R})$ such that $\emph{rank}(G) \leq k$ and $\sigma(A + G) = ( \beta_n, \ldots , \beta_1 )$ if and only if for each $1 \leq i \leq n - k$
    \begin{equation}
        \beta_{i + k} \leq \alpha_i \text{ and } \alpha_{i + k} \leq \beta_i. 
    \end{equation}
    \end{theorem}

    We now present the main result of the paper.
    
    \begin{theorem} \label{thm:main-k}
    Let $k, n$ be positive integers with $k < n$, and let $\mathcal{I}$ be a subset of $\{1,\ldots, n\}$ consisting of $n-k$ indices. 
    Consider two sets of positive numbers
    \begin{equation*}
      0 < \mu_1 \leq \cdots \leq \mu_n,
      \qquad
      0 < d_1 \leq \cdots \leq d_{n-k}
    \end{equation*}
    satisfying the interlacing inequalities \eqref{eq:gaussian-k-ineq}.
    Then there exists a $2n \times 2n$ real symmetric positive definite matrix $A$ with symplectic spectrum $\mu_1,\ldots, \mu_n$ such that its $s$-principal submatrix $A_\mathcal{I}$, as defined in \eqref{eq:sprincipal}, has symplectic spectrum $d_1,\ldots, d_{n-k}$.
    Moreover, $A$ can be chosen from the point-symplectic orbit of $\diag(\mu_n,\ldots, \mu_1) \oplus \diag(\mu_n,\ldots, \mu_1)$.
    \end{theorem}
    
\begin{proof}
    Set 
\begin{equation*}
   (\alpha_1,\ldots, \alpha_n) \coloneqq (\mu_n,\ldots,\mu_1)
  \text{ and }
   (\beta_1,\ldots, \beta_n) \coloneqq(d_{n-k},\ldots,d_1,\underbrace{0,\ldots, 0}_{k\text{ times}}).  
\end{equation*}
 The inequalities \eqref{eq:gaussian-k-ineq} are precisely
\begin{equation*}
    \alpha_i \geq \beta_{i+k},
  \qquad
  \beta_i \geq \alpha_{i+k},
  \qquad
  1 \leq i \leq n-k.  
\end{equation*}

Let $\Lambda \coloneqq \diag(\alpha_1,\ldots,\alpha_n)$.
By \cref{thm:paramount utility}, there exists $G \in M_n(\mathbb{R})$ with $\operatorname{rank} G \leq k$ such that
\begin{equation*}
    M \coloneqq \Lambda + G
\end{equation*}
has singular values $\beta_{n}, \ldots, \beta_1$.  
It follows that $M$ has nullity $k$.
Since $\Lambda$ is invertible, $\ker (\Lambda + G) \subseteq \operatorname{range}(\Lambda^{-1}G)$ and thus
\begin{equation*}
   k =  \operatorname{nullity}(\Lambda + G) \leq \operatorname{rank} G
\end{equation*}
forcing $\operatorname{rank} G=k$.
Now, factor $G \coloneqq -UV^{\top}$ with $U,V \in M_{n \times k}(\mathbb{R})$ having full column rank, which can be achieved using a singular value decomposition of $G$.
By \cref{lem:nullity-normalization},
\[
  V^{\top} \Lambda^{-1}U=I_k.
\]
Choose $F \coloneqq E_{\mathcal{I}^c}$, where $\mathcal{I}^c$ denotes $\{1,\ldots, n\} \backslash \mathcal{I}$, as defined by \eqref{eq:matrix-retaining-indices}.
\cref{lem:reconstruct-X} gives $X \in \mathrm{GL}_n(\mathbb{R})$ such that $U = \Lambda^{1/2}X^{\top}F$ and $ V = \Lambda^{1/2}X^{-1}F$.
Proposition \ref{prop:the connection} then identifies the non-zero singular values of $M$, which are $d_1, \ldots, d_{n-k}$, with the symplectic eigenvalues of the retained \emph{s}-principal submatrix $A_{\mathcal{I}}$ of
\begin{equation*}
   A \coloneqq (X \Lambda X^{\top})\oplus
  (X^{-\top} \Lambda X^{-1}). 
\end{equation*}
The global symplectic spectrum of $A$ is $\mu_1, \ldots, \mu_n$.
\end{proof}

\section{Constructive realization}
\label{sec:construction}

The proof of \cref{thm:main-k} also yields a constructive procedure for producing a realizing matrix. This follows by combining the numerical construction of Chu and Chu \cite{ChuChu2006} with \cref{lem:reconstruct-X,lem:nullity-normalization}. 
We briefly describe the procedure.

Suppose that
\begin{equation} \label{eq:algorithmic_realization}
  0<\mu_1\leq\cdots\leq\mu_n,
\qquad
0<d_1\leq\cdots\leq d_{n-k}  
\end{equation}
satisfy the symplectic interlacing inequalities \eqref{eq:gaussian-k-ineq}, and let $F\in M_{n\times k}(\mathbb R)$ consist of the standard basis vectors corresponding to the $k$ modes to be discarded.
Set
\begin{equation} \label{eq:algorithmic_realization_and_Lambda}
\alpha \coloneqq (\mu_n,\ldots,\mu_1),
\qquad
\beta \coloneqq (d_{n-k},\ldots,d_1,\underbrace{0,\ldots, 0}_{k\ {\rm times}}),
\qquad
\Lambda \coloneqq \operatorname{diag}(\alpha_1,\ldots,\alpha_n).     
\end{equation}

The construction proceeds as follows.

\begin{enumerate}
    \item \emph{Singular-value reassignment.}
    Apply the recursive low-rank singular value reassignment procedure of Chu and Chu \cite[Section~4]{ChuChu2006} to $\Lambda$ and the target singular values $\beta_n,\ldots, \beta_1$.
    Their procedure uses the rank-one reassignment algorithm as a building block and recursively introduces intermediate singular value lists, reducing the rank of the update by one at each stage.
    It produces a matrix
    \[
        G\in M_n(\mathbb R),\qquad \operatorname{rank} G\leq k,
    \]
    such that $M\coloneqq\Lambda+G$
    has singular values
    \[
        \underbrace{0,\ldots,0}_{k\ {\rm times}}, d_{1},\ldots,d_{n-k}.
    \]
    Since $M$ has nullity $k$ and $\Lambda$ is invertible, the argument in the proof of \cref{thm:main-k} shows that $\operatorname{rank}G=k$.

    \item \emph{Factorization of the low-rank defect.}
    Compute a compact singular value decomposition $G= U_1 \Sigma_1 V_1^{\top}$ where $U_1, V_1$ are $n \times k$ real matrices with orthonormal columns and $\Sigma_1$ is a $k \times k$ positive definite diagonal matrix.
    Define
    \[
        U \coloneqq U_1 \Sigma_1^{1/2},
        \qquad
        V \coloneqq - V_1 \Sigma_1^{1/2}.
    \]
    Then we have
    \[
        G=-UV^{\top},
        \qquad
        M=\Lambda-UV^{\top}.
    \]
    Since $\operatorname{nullity} M=k$, it may verified that $V^{\top}\Lambda^{-1}U=I_k$ is true in accordance with \cref{lem:nullity-normalization}.

    \item \emph{Reconstruction of the point transformation.}
    Set
    \[
        P \coloneqq \Lambda^{-1/2}U,
        \qquad
        Q \coloneqq \Lambda^{-1/2}V.
    \]
    Choose $R\in M_{(n-k)\times n}(\mathbb{R})$ whose rows form a basis of
    $\ker(Q^\top)$, and set
    \[
        Y\coloneqq
        \begin{bmatrix}
        R\\
        P^{\top}
        \end{bmatrix}.
    \]
    We get from the proof of \cref{lem:reconstruct-X} that
    \[
        YQ=
        \begin{bmatrix}
        0\\
        I_k
        \end{bmatrix}.
    \]
    Choose a permutation matrix $\Pi$ which sends the last $k$ ordered standard basis vectors to the columns of $F$, and define
    \[
        X \coloneqq \Pi Y.
    \]
    Observe that $X\in\operatorname{GL}_n(\mathbb R), U = \Lambda^{1/2} X^{\top} F$, and $V = \Lambda^{1/2} X^{-1} F$.
    This follows from \cref{lem:reconstruct-X}

    \item \emph{Construction of the realizing matrix.}
    Finally, form
    \[
        A_X \coloneqq
        (X \Lambda X^{\top})
        \oplus
        (X^{-\top} \Lambda X^{-1}).
    \]
    The global symplectic spectrum of $A_X$ is
    $\mu_1, \ldots, \mu_n$.
    By \cref{prop:the connection}, deleting the $k$ canonical pairs specified by $F$ produces an $s$-principal submatrix whose symplectic spectrum is $d_1, \ldots, d_{n-k}$.
\end{enumerate}

Chu and Chu noted that their basic rank-one implementation may require numerically more stable variants when singular values are clustered.
For such a situation, they recommended seeking remedy from \cite{de1978numerically, gragg1984numerically}. 
We also suggest the same literature for the numerical predicament that is inherited by the present construction.

We also present below a MATLAB program to obtain a realizing matrix $A_X$, with the pre-condition that the perturbation $G$ is given beforehand.
The program is developed using the algorithm discussed above.

\vspace{0.2cm}

\begin{tcolorbox}[
    colback=white,
    colframe=black,
    boxrule=0.6pt,
    arc=1mm,
    left=2mm,
    right=2mm,
    top=2mm,
    bottom=2mm
]
\begin{lstlisting}[style=matlabcode]
L = [enter manually];      % L represents the matrix Lambda
G = [enter manually];      % obtained from Chu and Chu (*@\cite{ChuChu2006}@*)
F = [enter manually];      % corresponding to deleted indices

[U_0,S_0,V_0] = svd(G);   % SVD of G
k = rank(G);

U_1 = U_0(:,1:k);
S_1 = S_0(1:k,1:k);
V_1 = V_0(:,1:k);         % G = U_1*S_1*V_1' is the compact SVD of G

U = U_1*sqrtm(S_1);
V = -V_1*sqrtm(S_1);

P = diag(1./sqrt(diag(L)))*U;
Q = diag(1./sqrt(diag(L)))*V;

R = null(Q')';
Y = [R; P'];

% Constructing the permutation matrix Pi
[n,k] = size(F);
[row_indices,~] = find(F == 1);
idx_F = row_indices(:)';
idx_rem = setdiff(1:n,idx_F);
p_vector = [idx_rem,idx_F];

Pi = eye(n);
Pi = Pi(:,p_vector);

X = Pi*Y;

A_X = blkdiag(X*L*X', (X'\L)/X)
\end{lstlisting}
\end{tcolorbox}

\begin{remark}
    The procedure above establishes constructivity in exact arithmetic; we do not claim that the resulting realization is optimally conditioned.
\end{remark}

\section{Application to the Gaussian quantum marginal problem}\label{sec:gaussian-application}

Consider an $n$-mode bosonic continuous-variable system with canonical quadratures ordered as
\[
  (\widehat q_1,\ldots,\widehat q_n,\widehat p_1,\ldots,\widehat p_n).
\]
For a centered Gaussian state, the state is specified by its symmetric covariance matrix $A \in  M_{2n}(\mathbb{R})$.  
In the convention used here, the Robertson--Schr\"odinger uncertainty relation is
\begin{equation}\label{eq:uncertainty}
  A+iJ\geq0,
\end{equation}
and this is equivalent to all symplectic eigenvalues of $A$ being at least one.
Moreover, every $2n \times 2n$ real positive definite matrix satisfying the uncertainty relation is a legitimate covariance matrix of an $n$-mode Gaussian state \cite{Serafini2023}.
A partial trace over a set of modes deletes the corresponding canonical coordinate pairs, so the covariance matrix of the reduced Gaussian state is exactly an $s$-principal submatrix of $A$; see \cite{Weedbrook2012,Serafini2023}.

    Our main result, combined with the symplectic interlacing theorem of Krbek et al.~thus settles the Gaussian quantum marginal problem of admissibility of global-versus-local symplectic spectra of Gaussian states and their marginals.  
    \begin{corollary}[Gaussian global-versus-subsystem marginal criterion]\label{cor:gaussian}
    Let $n$ and $k < n$ be positive integers.
    Suppose we have two sets of positive numbers
    \[
      1\leq\mu_1\leq\cdots\leq\mu_n,
      \qquad
      0 < d_1\leq\cdots\leq d_{n-k}.
    \]
    Then there exists an $n$-mode Gaussian state with global symplectic spectrum $\mu_1, \ldots, \mu_n$ whose retained $(n-k)$-mode subsystem has symplectic spectrum $d_1, \ldots, d_{n-k}$ if and only if the interlacing inequalities \eqref{eq:gaussian-k-ineq} are satisfied.
    \end{corollary}

\begin{proof}
Necessity is the symplectic interlacing theorem.  For sufficiency, \cref{thm:main-k} gives a positive definite matrix $A_X$ with the prescribed global and local symplectic spectra.  Because the global symplectic eigenvalues satisfy $\mu_j\geq1$, Williamson's criterion implies $A_X+iJ\geq0$.  Hence $A_X$ is a physical Gaussian covariance matrix, and its $s$-principal submatrix is the covariance matrix of the reduced Gaussian state.
\end{proof}

\cref{cor:gaussian} also admits a direct physical interpretation for multimode electromagnetic fields. In the Williamson normal form, the symplectic eigenvalues of a Gaussian covariance matrix determine the mean photon numbers $N_j$ of the corresponding thermal normal modes through $\mu_j=2N_j+1$, and similarly $d_j=2N_j'+1$ for the reduced subsystem. 
Hence, our result gives a complete characterization of the possible relations between the average photon numbers $(N_1,\ldots,N_n)$ of the global system and the average photon numbers $(N_1',\ldots,N_{n-k}')$ of an $(n-k)$-mode subsystem. 

    The algorithm developed in \cref{sec:construction} provides a way of preparing a global Gaussian state for given sets of global and local symplectic spectra.
     Starting from the Williamson covariance matrix \(\Lambda\oplus\Lambda\), which represents a product of thermal modes with the prescribed global symplectic eigenvalues.
     Here $\Lambda = \diag(\mu_n,\ldots,\mu_1)$.
     One then applies the Gaussian unitary corresponding to the point-symplectic transformation $X\oplus X^{-\mathsf T}$.
    The resulting covariance matrix is
    \[
        A_X=(X\Lambda X^{\mathsf T})
          \oplus
          (X^{-\mathsf T}\Lambda X^{-1}),
    \]
    and tracing out the selected modes leaves an \((n-k)\)-mode Gaussian state with the prescribed local symplectic spectrum.
  
\section{Conclusion}
    We established that the converse of the symplectic interlacing theorem holds. Combined with the symplectic interlacing inequalities of Krbek et al., our result gives necessary and sufficient compatibility conditions for a prescribed global symplectic spectrum and a local symplectic spectrum of an \(s\)-principal submatrix of a positive definite matrix.
    We also provide an algorithm to explicitly construct a positive definite matrix that realizes the given symplectic spectra.
    One of the novelties of our method is development of the connection between symplectic eigenvalues of $s$-principal submatrices of matrices in point symplectic orbit of $\Lambda$ and singular values of a low-rank update of $\Lambda$, which we believe will be valuable to matrix analysts and quantum information theorists.

\section*{Acknowledgments}
    Temjensangba thanks Nagaland University for granting study leave with pay.

\section*{AI declaration}
     The authors independently identified the research problem and made initial progress toward its solution. 
     A draft of our early progress was supplied to OpenAI's ChatGPT (GPT-5.6 Sol) and numerous interactions with it led to the development of the main technical parts of the paper.
     In particular, the main idea of linking the symplectic spectrum of a $2(n-k) \times 2(n-k)$ $s$-principal submatrix with singular values of a rank $k$ perturbed matrix was suggested by AI.
     The authors independently developed the algorithm for constructing a matrix that realizes the given spectra.
     The authors polished, simplified, and verified all the mathematical details, and
    they take full responsibility for the correctness, accuracy, and content of the manuscript.

\bibliographystyle{siamplain}
\bibliography{references}
\end{document}